\documentclass{article} 
\usepackage{iclr2027_conference,times}

\usepackage{url}

\title{Dimension-Free Decentralized Nonsmooth \\Nonconvex Stochastic Optimization}

\iclrfinalcopy

\author{Yuanyu Wan$^{1,2}$,~~Lan Xue$^{1,2}$,~~Haomin Bai$^3$,~~Tong Wei$^4$,~~Mingli Song$^{1,2}$\\
$^1$School of Software Technology, Zhejiang University \\
$^2$State Key Laboratory of Blockchain and Security, Zhejiang University\\
$^3$School of Artificial Intelligence, Nanjing University \\
$^4$School of Computer Science and Engineering,  Southeast University \\
\texttt{\{wanyy,xuelan,brooksong\}@zju.edu.cn}\\ 
\texttt{baihm@lamda.nju.edu.cn},~~\texttt{weit@seu.edu.cn}  
}

\usepackage{algorithm}
\usepackage{algorithmic}
\usepackage{amsfonts} 
\usepackage{amsmath}
\usepackage{amssymb}
\usepackage{mathtools}
\usepackage{amsthm}

\newtheorem{thm}{Theorem}
\newtheorem{prop}{Proposition}
\newtheorem{lem}{Lemma}
\newtheorem{cor}{Corollary}
\newtheorem{myDef}{Definition}
\newtheorem{assum}{Assumption}
\def \x {\mathbf{x}}
\def \y {\mathbf{y}}
\def \w {\mathbf{w}}
\def \z {\mathbf{z}}
\def \v {\mathbf{v}}
\def \g {\mathbf{g}}
\def \ze {\mathbf{0}}
\def \G {\mathcal{G}}
\def \u {\mathbf{u}}
\def \E {\mathbb{E}}
\def \R {\mathbb{R}}

\def \A {\mathcal{A}}
\DeclareMathOperator*{\argmin}{argmin}
\DeclareMathOperator*{\poly}{poly}

\begin{document}

\maketitle

\begin{abstract}
We investigate decentralized nonsmooth nonconvex stochastic optimization over a network of $n$
nodes, with the goal of finding an $(\delta,\epsilon)$-Goldstein stationary point. The best existing
algorithm achieves $O(\delta^{-1}(\epsilon^{-3}+d\epsilon^{-1}))$ sample complexity and
$\widetilde{O}(\gamma^{-1/2}\delta^{-1}(\epsilon^{-3}+d\epsilon^{-1}))$ communication complexity,
where $d$ is the problem dimension and $\gamma$ is the spectral gap of the communication matrix.
However, the polynomial dependence on $d$ can be a major bottleneck in high-dimensional regimes. In
this paper, we propose a novel algorithm that achieves $O(\delta^{-1}\epsilon^{-3})$ sample
complexity and $\widetilde{O}(\gamma^{-1/2}\delta^{-1}\epsilon^{-3})$ communication complexity. The
primary technique is an elegant decentralized online-to-nonconvex conversion that reduces the
original problem to a decentralized online convex optimization (D-OCO) problem. A key property of
our conversion is that its consensus requirements can be inherited directly from the consensus of
the underlying D-OCO decisions. In particular, this property enables us to establish an explicit
connection between the dimension dependence and the consensus error, which in turn shows that the
polynomial dependence on $d$ can be removed with only logarithmic additional communication.
\end{abstract}

\section{Introduction}
Decentralized stochastic optimization (DSO) \citep{Ram2010,Lian2017} has become an attractive
alternative to its centralized counterpart as modern learning problems increasingly involve
large-scale models and data distributed across multiple devices. Given a network of $n$ nodes and
the problem dimension $d$, DSO can be formulated as
\[
\min_{\x\in\R^d}
f(\x)=\frac{1}{n}\sum_{i=1}^n f_i(\x)
\]
where $f_i(\x)=\E_{\xi_i}\!\left[F_i(\x;\xi_i)\right]$ is the local objective held by node $i$, and
each node $i$ can only access the stochastic component $F_i(\cdot;\xi_i):\R^d\to\R$ indexed by a
random variable $\xi_i$ and communicate with its neighboring nodes. For convex objectives, the goal
is typically to find a solution with a small objective gap $f(\x)-f(\x^\ast)$ or a small error
$\|\x-\x^\ast\|$, where $\|\cdot\|$ denotes the Euclidean norm and
$\x^\ast\in\arg\min_{\x\in\R^d}f(\x)$. For nonconvex but smooth objectives, it is common to seek an
$\epsilon$-stationary point, i.e., $\|\nabla f(\x)\|\leq \epsilon$. These cases have been
extensively studied, leading to numerous algorithms with theoretical guarantees
\citep{Lan-2018,Choco-Gossip-ICML19,D-SGD-ICML20,Koloskova-NIPS2021,ADSGD-ICML21,Yuan-NIPS22}.

However, in many practical problems, such as training ReLU neural networks \citep{ReLU-ICML23}, the
objective function is often both nonconvex and nonsmooth. Motivated by these problems, we
investigate DSO with nonsmooth nonconvex objectives and aim to find an $(\delta,\epsilon)$-Goldstein
stationary point.\footnote{This is a tractable extension of the $\epsilon$-stationary point to
nonsmooth objectives and will be defined later.} Although a few recent studies
\citep{Zero-DSGD-Lin-2024,Sahinoglu-ICML24,Chen-Arxiv-2601} have proposed algorithms for this
setting, they remain unsatisfactory in terms of both sample and communication complexities.
Specifically, the best existing algorithm \citep{Chen-Arxiv-2601} achieves
$O(\delta^{-1}(\epsilon^{-3}+d\epsilon^{-1}))$ sample complexity and
$\widetilde{O}(\gamma^{-1/2}\delta^{-1}(\epsilon^{-3}+d\epsilon^{-1}))$ communication complexity,
where $\gamma$ is the spectral gap of the communication matrix.\footnote{The $\tilde{O}(\cdot)$
notation hides constant factors as well as polylogarithmic factors.} Such a polynomial dependence on
$d$ can limit the scalability of this algorithm in high-dimensional regimes. Thus, it is natural to
ask whether the polynomial dependence on $d$ can be further eliminated from both the sample and
communication complexities of DSO with nonsmooth nonconvex objectives.

In this paper, we provide an affirmative answer to the above question. Our primary technique is a
decentralized extension of the online-to-nonconvex conversion originally proposed for the
non-distributed setting \citep{Cutkosky2023}. Specifically, \citet{Cutkosky2023} reduce nonsmooth
nonconvex stochastic optimization to an  online convex optimization (OCO) problem
\citep{Online:suvery,hazan2016introduction}. The underlying OCO algorithm generates update
directions using losses constructed from stochastic gradients evaluated at random interpolation
points along the line segments connecting consecutive iterates. Inspired by \citet{Cutkosky2023},
our conversion reduces DSO with nonsmooth nonconvex objectives to a decentralized online convex
optimization (D-OCO) problem \citep{wan2024nearly,wan2024optimal}. The underlying D-OCO algorithm
generates update directions for each node using losses constructed from local stochastic gradients
evaluated at random interpolation points along the line segments connecting consecutive local
iterates. Compared with \citet{Cutkosky2023}, our key additional challenges are to control the
consensus among the local iterates and random interpolation points across different nodes, as well
as the discrepancy between the average of local gradients and the global gradient evaluated at the
average random interpolation point.

To this end, we first observe that the consensus error of the local iterates is proportional to that
of the local update directions generated by the D-OCO algorithm. In contrast, the consensus error of
the random interpolation points does not generally inherit this property, because different nodes
may use different interpolation weights. Interestingly, we show that this issue can be addressed by
using a shared random interpolation weight across nodes, which requires only a common random seed.
Furthermore, to control the gradient discrepancy error, we apply the classical randomized smoothing
technique \citep{Duchi-SIAM12} to the original objective functions. This yields an error bound
proportional to the product of the consensus error of the random interpolation points and the
smoothness constant induced by randomized smoothing. Although such a smoothness constant scales as
$\sqrt{d}$, our preceding consensus analysis allows us to cancel this factor by sufficiently
tightening the consensus error of the underlying D-OCO decisions. Taken together, we reduce all
additional challenges introduced by decentralization to controlling the consensus error of the
underlying D-OCO decisions. Importantly, existing D-OCO algorithms such as \citet{wan2024nearly} can
achieve the required consensus error with only logarithmic additional communication. Finally, by
combining our conversion with the D-OCO algorithm of \citet{wan2024nearly}, we obtain a novel
algorithm with $O(\delta^{-1}\epsilon^{-3})$ sample complexity and
$\widetilde{O}(\gamma^{-1/2}\delta^{-1}\epsilon^{-3})$ communication complexity.

Note that two existing algorithms \citep{Sahinoglu-ICML24,Chen-Arxiv-2601} can also be interpreted
as decentralized online-to-nonconvex conversions and have employed the randomized smoothing
technique. However, the consensus among their local iterates is guaranteed by further applying the
standard or accelerated gossip \citep{Xiao-Gossip04,Acc_Gossip} to the iterates themselves. This
incurs additional communication outside the D-OCO algorithm and complicates the analysis. More
critically, without the shared random interpolation weight used in our conversion, the consensus
error of their random interpolation points contains an \(O(1)\) term. Consequently, the \(\sqrt{d}\)
factor appearing in their gradient discrepancy error cannot be canceled regardless of how small the
consensus error of the underlying D-OCO decisions is. In contrast, our conversion provides a clean
and unified way to exploit the consensus property of the underlying D-OCO decisions and allows us to
achieve complexity bounds without polynomial dependence on \(d\) for the first time.

\section{Related Work}
In this section, we briefly review related work on nonsmooth nonconvex stochastic optimization in
both the non-distributed and decentralized settings.

\subsection{Non-distributed Nonsmooth Nonconvex Stochastic Optimization}
Early studies on non-distributed nonsmooth nonconvex stochastic optimization mainly focus on
establishing asymptotic convergence guarantees
\citep{Kiwiel07,Majewski18-Arxiv,Davis-FOCM20,Bolte-MP21}. The pioneering work of
\citet{Zhang2020nonsmooth} provides the first non-asymptotic analysis for this problem. Note that
for nonsmooth objectives, \citet{Zhang2020nonsmooth} show that the \(\epsilon\)-stationary point cannot be found in finite time. To this end, they introduce
a tractable relaxation, namely the $(\delta,\epsilon)$-Goldstein stationary point, and establish the
first finite-time guarantee for this notion, with $O(\delta^{-1}\epsilon^{-4})$ sample complexity.
Moreover, under a deterministic generalized gradient oracle, they obtain an improved oracle
complexity of $O(\delta^{-1}\epsilon^{-3})$. Since then, this stationarity criterion has attracted
growing research interest, particularly in the deterministic setting
\citep{Davis2022gradient,Tian2022finite,Kornowski2022small,%
Jordan2023deterministic,Tian2024dimension}.
The first improvement in the stochastic setting is achieved by \citet{Cutkosky2023}, who reduce the
sample complexity to $O(\delta^{-1}\epsilon^{-3})$ and show the optimality of this rate. The key
technique of \citet{Cutkosky2023} is a novel online-to-nonconvex conversion that employs an OCO
algorithm \citep{hazan2016introduction} to generate update directions rather than the iterates
themselves. Specifically, the losses fed to the underlying OCO algorithm are constructed from
stochastic gradients evaluated at random interpolation points along the line segments connecting
consecutive iterates, and the final output is also formed from these interpolation points.

In addition, the more challenging zeroth-order setting, where only function values are available,
has been studied recently \citep{Lin-2022,Chen-ICML23,Kornowski-JMLR24}. Specifically, in this
setting, the randomized smoothing technique \citep{Duchi-SIAM12} serves as a natural bridge from
zeroth-order information to gradient information. Meanwhile, it transforms the original nonsmooth
objective into a smooth surrogate, at the price of introducing a $\sqrt{d}$ factor into the
smoothness constant. By applying randomized smoothing and exploiting the smoothness of the surrogate
objective, \citet{Lin-2022} first propose a gradient-free method with
$O(d^{3/2}\delta^{-1}\epsilon^{-4})$ zeroth-order sample complexity. Later, \citet{Chen-ICML23}
further improve the zeroth-order sample complexity to $O(d^{3/2}\delta^{-1}\epsilon^{-3})$ by using
a variance reduction technique \citep{SPIDER-2018}. More interestingly, \citet{Kornowski-JMLR24}
establish a connection between the Goldstein subdifferentials of the original and smoothed
objectives, which allows them to combine randomized smoothing with the online-to-nonconvex
conversion of \citet{Cutkosky2023}, rather than relying on the smoothness of the surrogate
objective. This yields the optimal zeroth-order sample complexity of $O(d\delta^{-1}\epsilon^{-3})$.

\subsection{Decentralized Nonsmooth Nonconvex Stochastic Optimization}
For decentralized nonsmooth nonconvex stochastic optimization,
\citet{Zero-DSGD-Lin-2024} provide the first finite-time analysis under the zeroth-order setting. They propose a decentralized gradient-free
method (DGFM) and its improved variant called DGFM$^{+}$ by extending the algorithms of
\citet{Lin-2022} and \citet{Chen-ICML23}, respectively. Specifically, DGFM achieves
$O(nd^{3/2}\poly(1/\gamma)\delta^{-1}\epsilon^{-4})$
zeroth-order sample complexity and
$O(d^{3/2}\poly(1/\gamma)\delta^{-1}\epsilon^{-4})$
communication complexity, whereas DGFM$^{+}$ achieves
$O(n^{1/2}d^{1/2}\delta^{-1}
(n\epsilon^{-2}+d\epsilon^{-3}))$
zeroth-order sample complexity and
$O(n^{1/2}d^{1/2}\poly(1/\gamma)\delta^{-1}\epsilon^{-2})$
communication complexity. As in the non-distributed case, the polynomial dependence of these
complexity bounds on $d$ is intrinsic to zeroth-order optimization. Later, \citet{Sahinoglu-ICML24}
consider the first-order setting and propose a multi-epoch decentralized online learning algorithm
(ME-DOL) by extending the online-to-nonconvex conversion of \citet{Cutkosky2023}. They adopt an
existing D-OCO algorithm of \citet{D-OMD-TAC17} to generate update directions for each node and
encounter the three aforementioned challenges introduced by decentralization.

For the first challenge, i.e., controlling the consensus among the local iterates, they further
apply one step of standard gossip \citep{Xiao-Gossip04} to the iterates after each update. For the
second challenge, i.e., controlling the consensus of the random interpolation points, no additional
mechanism is introduced. Instead, the resulting bound consists of the consensus error of the local
iterates together with an additional $O(1)$ term. For the third challenge, i.e., controlling the
gradient discrepancy error, they apply randomized smoothing \citep{Duchi-SIAM12} to the original
objective functions and exploit the smoothness of the resulting surrogate functions together with
the consensus bound for the random interpolation points. As a consequence of these consensus
properties  and the $\sqrt{d}$ factor in the smoothness constant, ME-DOL achieves only
$O(n\delta^{-1}\epsilon^{-3}(n\gamma^{-2}+dn^{1/2}\gamma^{-1}))$
sample complexity and
$O(\delta^{-1}\epsilon^{-3}(n\gamma^{-2}+dn^{1/2}\gamma^{-1}))$
communication complexity. Note that these bounds do not exhibit a clear advantage over those of
DGFM$^{+}$, which operates in the zeroth-order setting.

To address this limitation, \citet{Chen-Arxiv-2601} propose an improved algorithm called
decentralized online-to-nonconvex conversion with client sampling (DOC$^2$S). Their key idea is to
replace the single standard gossip step \citep{Xiao-Gossip04} used in ME-DOL with multiple
accelerated gossip steps \citep{Acc_Gossip}, which provide sharper consensus guarantees with respect
to the dependence on $n$ and $\gamma$. Moreover, they incorporate a client sampling technique
\citep{Chen-TMLR22} that reduces the sample complexity by a factor of $n$, since only one node
queries a local stochastic gradient at each iteration. Building on these two improvements, DOC$^2$S
achieves $O(\delta^{-1}(\epsilon^{-3}+d\epsilon^{-1}))$ sample complexity and
$\widetilde{O}(\gamma^{-1/2}\delta^{-1}(\epsilon^{-3}+d\epsilon^{-1}))$ communication complexity,
thereby exhibiting a clear benefit from access to first-order information. Nonetheless, the
polynomial dependence on $d$ remains unsatisfactory.

\section{Preliminaries}
In this section, we introduce the necessary preliminaries including the problem setup, randomized
smoothing, and D-OCO.

\subsection{Problem Setup}
Following previous studies \citep{Sahinoglu-ICML24,Chen-Arxiv-2601}, we consider a decentralized
network whose communication topology is modeled by an undirected graph $\G=([n],E)$, where
$[n]=\{1,2,\ldots,n\}$ denotes the set of nodes and $E\subseteq[n]\times[n]$ denotes the set of
edges. Each communication round among the nodes occurs via a gossip step \citep{Xiao-Gossip04},
i.e., computing a weighted average of some local variables according to a weight matrix
$P\in\mathbb{R}^{n\times n}$. It is common to impose the following assumption on $P$, and the
spectral gap of $P$ is then given by $\gamma=1-\sigma_2(P)\in(0,1]$, where $\sigma_{2}(P)$ denotes
the second largest singular value of $P$.
\begin{assum}
\label{assum:p}
The communication matrix $P \in \mathbb{R}^{n \times n}$ is supported on the graph $\G=([n], E)$,
and doubly stochastic, which satisfies: (i) $P_{ij}>0$ only if $(i, j) \in E$ or $i=j$; (ii)
$\sum_{j=1}^{n} P_{i j}=1, \forall i \in[n]$, and $\sum_{i=1}^{n} P_{i j}=1, \forall j \in[n]$.
Moreover, $P$ is symmetric and positive semidefinite, and $\sigma_{2}(P)<1$. 
\end{assum}

Besides, there are also several standard assumptions on the nonsmooth nonconvex objectives.
\begin{assum}
\label{assum:lower-bounded}
The global objective $f(\x)$ is lower bounded, i.e., $
f^\star=\inf_{\x\in\R^d}f(\x)>-\infty.$
\end{assum}
\begin{assum}
\label{assum:lip}
For each node $i\in[n]$, the stochastic component
$F_i(\x;\xi_i)$ is $L(\xi_i)$-Lipschitz, i.e., $
|F_i(\x;\xi_i)-F_i(\y;\xi_i)|
\leq
L(\xi_i)\|\x-\y\|,
\forall\,\x,\y\in\R^d$.
Moreover, the random Lipschitz constant has a bounded second moment, i.e., there exists $L>0$ such
that $\E_{\xi_i}[L(\xi_i)^2]\leq L^2,
\forall\, i\in[n].$
\end{assum}
\begin{assum}
\label{assum:oracle}
Each node $i\in[n]$ has access to a local stochastic
gradient oracle that returns $\nabla F_i(\x;\xi_i)$ for a given $\x$. The oracle is unbiased, i.e.,
$
\E_{\xi_i}[\nabla F_i(\x;\xi_i)]
=
\nabla f_i(\x)$. Moreover, there exists a constant $G>0$ such that  $
\E_{\xi_i}\!\left[
\|\nabla F_i(\x;\xi_i)\|^2
\right]
\leq G^2$.
\end{assum}
\textbf{Remark.} Assumption \ref{assum:lip} directly implies that each local objective $f_i(\x)$ is
$L$-Lipschitz, which will be used in our analysis. Moreover, the sample-wise Lipschitz assumption on
$F_i(\x;\xi_i)$ ensures that it is differentiable almost everywhere, so that the stochastic gradient
oracle in Assumption \ref{assum:oracle} is well-defined almost everywhere.

Under these assumptions, our goal is to find an $(\delta,\epsilon)$-Goldstein stationary point
\citep{Zhang2020nonsmooth} of the global objective $f(\x)$. To be precise, for a Lipschitz function
$f(\cdot):\R^d\mapsto\R$, let $\partial f(\x)$ denote its Clarke subdifferential
\citep{Clarke-1990}. The Goldstein $\delta$-subdifferential \citep{Goldstein1977} is defined as
$\partial_\delta f(\x)
=
\operatorname{conv}(
\cup_{\|\y-\x\|\leq\delta}
\partial f(\y)
)$, where $\operatorname{conv}(\cdot)$ denotes the convex hull of the given set. Then, the
$(\delta,\epsilon)$-Goldstein stationary point can be formally defined as below.
\begin{myDef}
\label{def:goldstein}
A point $\x\in\R^d$ is called $(\delta,\epsilon)$-Goldstein stationary for $f(\cdot):\R^d\mapsto\R$
if $\|\nabla f(\x)\|_\delta\leq \epsilon
$, where $\|\nabla f(\x)\|_\delta=
\min\{\|\g\|:\g\in\partial_\delta f(\x)\}$.
\end{myDef}

\subsection{Randomized Smoothing}
Randomized smoothing is a classical technique for constructing a smooth surrogate of a
nonsmooth function \citep{Duchi-SIAM12}. Let $f(\cdot):\R^d\to\R$ denote the function and
let $\mathcal{B}^d=\{\u\in\R^d:\|\u\|\leq 1\}$ denote the unit Euclidean ball centered at the
origin of $\R^d$. For any smoothing radius $\delta>0$, the smoothed version of $f(\x)$ is
defined as
\begin{equation}
    f_\delta(\x)
    =\E_{\u\sim\operatorname{Unif}(\mathcal{B}^d)}[f(\x+\delta\u)]
    \label{eq:smoothing-definition}
\end{equation}
where $\operatorname{Unif}(\mathcal{B}^d)$ denotes the uniform distribution on
$\mathcal{B}^d$. As summarized by \citet{Lin-2022}, the smoothed 
function has the following basic properties.
\begin{lem}[Proposition 2.3 of \citet{Lin-2022}]
\label{lem:smoothing}
Suppose that $f(\cdot):\R^d\to\R$ is $L$-Lipschitz. Then, for any $\delta>0$, its smoothed
version defined in \eqref{eq:smoothing-definition} satisfies:
(i) $|f_\delta(\x)-f(\x)|\leq \delta L$ for all $\x\in\R^d$;
(ii) $f_\delta(\x)$ is differentiable and $L$-Lipschitz; (iii) $f_\delta(\x)$ has
$(c\sqrt{d}L\delta^{-1})$-Lipschitz gradients for some constant $c>0$.
\end{lem}
Lemma \ref{lem:smoothing} allows us to convert the original nonsmooth problem into a smooth one 
and thereby exploit algorithmic and analytical techniques developed for smooth optimization. 
However, this alone only yields stationarity guarantees for the smoothed surrogate. 
To transfer such guarantees back to the original nonsmooth objective, 
an additional connection between their stationarity measures is required. 
Fortunately, \citet{Kornowski-JMLR24} have established the following lemma.
\begin{lem}[Lemma 4 of \citet{Kornowski-JMLR24}]
\label{lem:smoothing-goldstein}
Suppose that $f(\cdot):\R^d\to\R$ is Lipschitz. Then, for any
$\delta,r\geq0$ and $\x\in\R^d$, its smoothed version defined in
\eqref{eq:smoothing-definition}, with the convention $f_0=f$, satisfies
$\partial_r f_\delta(\x)\subseteq\partial_{r+\delta}f(\x)$.
\end{lem}
As discussed in \citet{Sahinoglu-ICML24}, for any $\delta>0$ and
$a\in(0,1)$, Lemma \ref{lem:smoothing-goldstein} can be simply used to derive that
\begin{equation}
    \|\nabla f(\x)\|_\delta
    \leq \|\nabla f_{a\delta}(\x)\|_{(1-a)\delta},\forall\x\in\R^d.
    \label{eq:smoothing-stationarity}
\end{equation}
Consequently, any $((1-a)\delta,\epsilon)$-Goldstein stationary point of $f_{a\delta}(\cdot)$ is
also an $(\delta,\epsilon)$-Goldstein stationary point of $f(\cdot)$. For simplicity, we set $a=1/2$ throughout the paper. 
Therefore, it suffices to find an $(\delta/2,\epsilon)$-Goldstein stationary point of the $(\delta/2)$-smoothed objective.

\subsection{Decentralized Online Convex Optimization (D-OCO)}
D-OCO is formulated as a collaborative game between $n$ decentralized nodes and an adversary, and
has been extensively studied in the literature
\citep{DAOL_TKDE,D-OMD-TAC17,D-OCO-Survey,wan2024nearly,wan2024optimal}.
Specifically, at each round $t\in[T]$, each node $i$ first selects a decision
$\mathbf{x}_{t,i}\in\mathcal{K}$, where $\mathcal{K}\subseteq\R^d$ is a convex set, and then
observes a convex local loss $\ell_{t,i}:\mathcal{K}\to\R$. Let
$\ell_t(\mathbf{x})=\sum_{j=1}^n\ell_{t,j}(\mathbf{x})$
denote the global loss at round $t$. The goal of each node $i$ is to minimize
its regret measured in terms of the global losses, i.e.,
\begin{equation*}
    \operatorname{Reg}_{T,i}
    =\sum_{t=1}^T\ell_t(\mathbf{x}_{t,i})
    -\min_{\mathbf{x}\in\mathcal{K}}\sum_{t=1}^T\ell_t(\mathbf{x}).
\end{equation*}
As in DSO, each node $i$ can communicate only with its immediate neighbors through gossip steps
\citep{Xiao-Gossip04}. Note that the D-OCO algorithm of \citet{wan2024nearly} has achieved a nearly
optimal regret bound together with an explicit upper bound on the consensus error, i.e.,
$\|\x_{t,i}-\bar{\x}_{t}\|$ for all $i\in[n]$, where
$\bar{\x}_t=(1/n)\sum_{j=1}^n\x_{t,j}$. Both guarantees will be used in our decentralized
online-to-nonconvex conversion.

\section{Our Decentralized Online-to-Nonconvex Conversion}
In this section, we first introduce the procedure of our decentralized online-to-nonconvex conversion (D-ONC), 
and then show how to instantiate it with the D-OCO algorithm of \citet{wan2024nearly} to obtain dimension-free sample 
and communication complexities. Omitted proofs can be found in the appendix.

\subsection{Detailed Procedure}
To aid understanding, we begin with the special case of a single node, 
which can be handled directly by the online-to-nonconvex conversion of \citet{Cutkosky2023}. 
Specifically, this conversion proceeds over $K$ epochs, each consisting of $T$ iterations. 
Let $\x_{T,1}^0=\ze\in\R^d$ denote an initial point. During each epoch $k$, it is natural to 
maintain the iterate $\x_{t,1}^k$ by setting $\x_{0,1}^k=\x_{T,1}^{k-1}$ and performing the 
update $\x_{t,1}^k=\x_{t-1,1}^k+\Delta_{t,1}^k$, 
where $\Delta_{t,1}^k$ is the update direction. Critically, \citet{Cutkosky2023} propose
to generate the update direction via an OCO algorithm. The key insight is that 
if $f_1(\cdot)$ is well-behaved, 
one can construct a stochastic gradient
$\g_{t,1}^k=\nabla F_1(\w_{t,1}^k;\xi_{t,1}^k)$ such that 
\begin{equation}
f_1(\x_{t,1}^k)-f_1(\x_{t-1,1}^k)=\E[\langle\g_{t,1}^k,\Delta_{t,1}^k\rangle]
\label{eq:loss_gap}
\end{equation}
by setting $\w_{t,1}^k=\x_{t-1,1}^k+s_{t,1}^k\Delta_{t,1}^k$, where $s_{t,1}^k\sim\operatorname{Unif}[0,1]$.
This gives rise to an OCO problem: $\Delta_{t,1}^k$ should incur a small linear loss
$\langle \g_{t,1}^k,\Delta_{t,1}^k\rangle$,
while $\g_{t,1}^k$ is revealed only after $\Delta_{t,1}^k$ has been chosen.
Thus, it is natural to generate $\{\Delta_{t,1}^k\}_{t\in[T]}$
by running an OCO algorithm on the sequence of linear losses
$\{
\ell_{t,1}^k(\x)
=
\langle \g_{t,1}^k,\x\rangle
\}_{t\in[T]}$. In addition, following \citet{Cutkosky2023}, the final solution should be $\w_1\sim\operatorname{Unif}\{\w_{1}^1,\ldots,\w_{1}^K\}$, where $\w_{1}^k=(1/T)\sum_{t=1}^T\w_{t,1}^k$ for each epoch $k$. 

Then, we extend the above construction to the general case with $n>1$ nodes. 
The key idea is conceptually simple: each node $i$ maintains 
its own local copies $\x_{t,i}^k$, $\Delta_{t,i}^k$, $\w_{t,i}^k$, 
and $\g_{t,i}^k$ of the corresponding quantities introduced above. 
Specifically, for any node $i\in[n]$, we initialize $\x_{T,i}^0=\ze\in\R^d$, 
and maintain the iterate $\x_{t,i}^k$ during each epoch $k\in[K]$ by setting
$\x_{0,i}^k=\x_{T,i}^{k-1}$ and performing the 
update $\x_{t,i}^k=\x_{t-1,i}^k+\Delta_{t,i}^k$ with some update direction $\Delta_{t,i}^k$. 
Note that although the update of $\x_{t,i}^k$ follows the same form as in the single-node case, 
the generation of $\Delta_{t,i}^k$ is fundamentally different due to the mismatch between 
the local objective $f_i(\x)$ and the global objective $f(\x)$. To correct this mismatch, 
we first introduce a virtual global variable $\bar{\x}_t^k=(1/n)\sum_{j=1}^n\x_{t,j}^k$ 
and exploit the Lipschitz continuity of each local objective $f_i$ to obtain
\begin{equation}
\begin{split}
f(\bar{\x}_t^k)-f(\bar{\x}_{t-1}^k)=&\frac{1}{n}\sum_{j=1}^n\left(f_j(\x_{t,j}^k)-f_j(\x_{t-1,j}^k)+f_j(\bar{\x}_t^k)-f_j(\x_{t,j}^k)+f_j(\x_{t-1,j}^k)-f_j(\bar{\x}_{t-1}^k)\right)\\
=&\frac{1}{n}\sum_{j=1}^n\E[\langle\g_{t,j}^k,\Delta_{t,j}^k\rangle]+O\left(\max_{j\in[n]}\|\x_{t,j}^k-\bar{\x}_t^k\|+\max_{j\in[n]}\|\x_{t-1,j}^k-\bar{\x}_{t-1}^k\|\right)
\end{split}
\label{eq:dec_loss_gap}
\end{equation}
as long as $\g_{t,i}^k$ satisfies the node-wise counterpart of \eqref{eq:loss_gap} for every $i\in[n]$ (i.e., replacing the index $1$ with $i$).
If only considering the first term on the right-hand side of \eqref{eq:dec_loss_gap}, one can simply run an
independent OCO algorithm at each node $i\in[n]$ to generate $\Delta_{t,i}^k$. 
However, controlling the second term on the right-hand side of \eqref{eq:dec_loss_gap} further requires sufficient
consensus among the local iterates, which cannot be guaranteed by independent OCO algorithms. To address this issue, we define another virtual global variable $\bar{\Delta}_{t}^k=(1/n)\sum_{j=1}^n\Delta_{t,j}^k$, 
and obtain
\begin{equation}
\begin{split}
\frac{1}{n}\sum_{j=1}^n\E[\langle\g_{t,j}^k,\Delta_{t,j}^k\rangle]=\frac{1}{n}\sum_{j=1}^n\E[\langle\g_{t,j}^k,\bar{\Delta}_{t}^k\rangle]+O\left(\max_{j\in[n]}\|\Delta_{t,j}^k-\bar{\Delta}_{t}^k\|\right).
\end{split}
\label{eq:local_to_global}
\end{equation}
This motivates us to generate $\{\Delta_{t,i}^k\}_{t\in[T],i\in[n]}$ by running a D-OCO algorithm on the sequence of local linear losses
$\{
\ell_{t,i}^k(\x)
=
\langle \g_{t,i}^k,\x\rangle
\}_{t\in[T],i\in[n]}$. Indeed, the two terms on the right-hand side of \eqref{eq:local_to_global} can be 
controlled by the regret and consensus error bounds of the D-OCO algorithm, respectively. More interestingly,
by combining this observation with the update rule of $\x_{t,i}^k$, we notice that the second term on the right-hand side of \eqref{eq:dec_loss_gap}
can be controlled simultaneously as long as the consensus error bound of the D-OCO algorithm is sufficiently small.

The remaining problem is how to construct the stochastic gradient $\g_{t,i}^k$. Inspired by the single-node case, a naive idea is to independently sample $s_{t,i}^k\sim\operatorname{Unif}[0,1]$
, and set $\w_{t,i}^k=\x_{t-1,i}^k+s_{t,i}^k\Delta_{t,i}^k$ and $\g_{t,i}^k=\nabla F_i(\w_{t,i}^k;\xi_{t,i}^k)$. 
Moreover, the final solution $\w_1$ in the single-node case can be naturally extended to $\bar{\w}\sim\operatorname{Unif}\{\bar{\w}^1,\ldots,\bar{\w}^K\}$,
where $\bar{\w}^k=(1/(Tn))\sum_{t=1}^T\sum_{i=1}^n\w_{t,i}^k$ for each epoch $k$. 
However, establishing $(\delta,\epsilon)$-Goldstein stationarity requires relating the
average of local gradients to gradients of the global objective
evaluated at common points across nodes.
Specifically, let $\bar{\w}_t^k=(1/n)\sum_{i=1}^n\w_{t,i}^k$.
In the analysis, we can directly control
$(1/(nT))\sum_{t=1}^T\sum_{i=1}^n\nabla f_i(\w_{t,i}^k)$,
whereas the relevant quantity for $(\delta,\epsilon)$-Goldstein stationarity is
$(1/T)\sum_{t=1}^T\nabla f(\bar{\w}_t^k)$,
provided that the points $\bar{\w}_t^k$ remain sufficiently close
to the output $\bar{\w}^k$.
These two averages can differ because the local gradients are
evaluated at different points.
For nonsmooth objectives, this discrepancy cannot generally be
controlled by the distances between these points.
We therefore apply randomized smoothing and carry out the analysis
with the smoothed objectives.
Their smoothness allows us to bound the corresponding gradient
discrepancy in terms of the consensus errors
$\|\w_{t,i}^k-\bar{\w}_t^k\|$. Now, a critical limitation of the previous naive idea becomes apparent: 
the independent samples $\{s_{t,i}^k\}_{i\in[n]}$ introduce an $O(1)$ term in the resulting consensus bound for $\{\w_{t,i}^k\}_{i\in[n]}$, 
even when the consensus errors of $\{\x_{t-1,i}^k\}_{i\in[n]}$ and  $\{\Delta_{t,i}^k\}_{i\in[n]}$ are arbitrarily small.  
To address this limitation, we instead use a shared random interpolation weight $s_t^k\sim\operatorname{Unif}[0,1]$ across all nodes,
and set $\w_{t,i}^k=\x_{t-1,i}^k+s_{t}^k\Delta_{t,i}^k$.
Despite its simplicity, this change allows the consensus error of
$\{\w_{t,i}^k\}_{i\in[n]}$ to be directly controlled by those of
$\{\x_{t-1,i}^k\}_{i\in[n]}$ and $\{\Delta_{t,i}^k\}_{i\in[n]}$, and hence
ultimately by the consensus error of the underlying D-OCO algorithm according
to the above discussions. In addition, to reduce the total sample complexity, 
we also incorporate the client sampling technique \citep{Chen-TMLR22,Chen-Arxiv-2601}, 
i.e., only one node $i_t^k\sim\operatorname{Unif}([n])$ is sampled at each iteration $t$ of epoch $k$ 
and queries its local stochastic gradient. To keep the unbiasedness of $\g_{t,i}^k$ (and recalling the randomized smoothing technique), 
we construct it as
\begin{equation*}
\g_{t,i}^k
=
\begin{cases}
n\nabla F_i\!\left(
\w_{t,i}^k+\delta^\prime \u_{t,i}^k;\xi_{t,i}^k
\right),
& i=i_t^k\\[1mm]
\ze,
& i\neq i_t^k
\end{cases}
\end{equation*}
where $\u_{t,i}^k\sim\operatorname{Unif}(\mathcal{B}^d)$ and $\delta^\prime=\delta/2$ is
the smoothing radius.

\begin{algorithm}[t]
\caption{Decentralized Online-to-Nonconvex Conversion (D-ONC)}
\label{alg:donc}
\begin{algorithmic}[1]
\STATE \textbf{Input:} Number of epochs $K$, epoch length $T$, stationarity radius $\delta$, and a D-OCO algorithm $\A$ with a decision set
$\mathcal{K}$
\STATE Set $\x_{T,i}^{0}=\ze$ for all $i\in[n]$
\FOR{$k=1,\ldots,K$}
  \STATE Restart $\A$ over $\mathcal{K}$ and set
  $\x_{0,i}^{k}=\x_{T,i}^{k-1}$ for all $i\in[n]$
  \FOR{$t=1,\ldots,T$}
    \STATE Obtain $\Delta_{t,i}^{k}$ from $\A$ and set
    $\x_{t,i}^{k}=\x_{t-1,i}^{k}+\Delta_{t,i}^{k}$ for all $i\in[n]$
    \STATE Draw shared $s_t^k\sim\operatorname{Unif}[0,1]$ and set
    $\w_{t,i}^{k}=\x_{t-1,i}^{k}+s_t^k\Delta_{t,i}^{k}$ for all $i\in[n]$
    \STATE Draw $i_t^k\sim\operatorname{Unif}([n])$,
    $\u_{t,i_t^k}^k\sim\operatorname{Unif}(\mathcal{B}^d)$, and $\xi_{t,i_t^k}^k$
    \STATE Set $\g_{t,i_t^k}^k=n\nabla F_{i_t^k}
    \left(\w_{t,i_t^k}^k+(\delta/2)\u_{t,i_t^k}^k;\xi_{t,i_t^k}^k\right)$ and
    $\g_{t,i}^k=\ze$ for all $i\neq i_t^k$
    \STATE Send $\ell_{t,i}^k(\x)=\langle\g_{t,i}^k,\x\rangle$ to the local node $i$ of
    $\A$, for all $i\in[n]$
  \ENDFOR
  \STATE Set $\bar{\w}^{k}=(nT)^{-1}\sum_{t=1}^T\sum_{i=1}^n\w_{t,i}^{k}$
\ENDFOR
\RETURN $\bar{\w}\sim\operatorname{Unif}\{\bar{\w}^{1},\ldots,\bar{\w}^{K}\}$
\end{algorithmic}
\end{algorithm}

Based on the above discussions, the detailed procedure of our conversion is summarized in Algorithm \ref{alg:donc},
where the D-OCO algorithm is denoted as $\A$ and its decision set $\mathcal{K}$ will be specified later. 

\subsection{Theoretical Guarantees}
To demonstrate the power of our conversion, we first establish a general reduction guarantee
in terms of the regret and consensus error of the D-OCO algorithm. Let $\mathcal{R}_T$ denote a uniform deterministic upper bound on the expected regret
of all nodes over all epochs, and let $\mathcal{C}_{T,K}$ denote a uniform deterministic upper bound on 
the consensus error of all nodes over all iterations and epochs, i.e.,
\begin{equation}
\mathcal{R}_T\geq \max_{i,k}\E\left[\sum_{t=1}^T\sum_{j=1}^n\ell_{t,j}^k(\Delta_{t,i}^k)-\min_{\u\in\mathcal{K}}\sum_{t=1}^T\sum_{j=1}^n\ell_{t,j}^k(\u)\right],\quad \mathcal{C}_{T,K}\geq \max_{i,t,k}\E[\|\Delta_{t,i}^k-\bar{\Delta}_t^k\|].
\label{two-definitions}
\end{equation}
The final solution of our conversion satisfies the following theorem.
\begin{thm}
\label{thm:general-reduction}
Suppose that Assumptions~\ref{assum:lower-bounded}--\ref{assum:oracle} hold. Following the definitions in \eqref{two-definitions}, for any $\delta>0$, Algorithm~\ref{alg:donc} with
$\mathcal{K}=\{\x\in\R^d:\|\x\|\leq D\}$ and $D=\delta/(2T)$ satisfies
\begin{equation}
  \E\|\nabla f(\bar{\w})\|_{\delta}
  \leq
  \frac{\mathcal{R}_T}{nDT}
  +\frac{G}{\sqrt T}
  +\frac{2(\Delta_f+L\delta)}{\delta K}
  +\frac{L(K+1)\mathcal{C}_{T,K}}{D}
  +\frac{c\sqrt{d}L(KT+1)\mathcal{C}_{T,K}}{\delta}
  \label{eq:general-reduction}
\end{equation}
where $\Delta_f=f(\ze)-f^\star$ and $c>0$ is the numerical constant in
Lemma~\ref{lem:smoothing}. 
\end{thm}

\textbf{Remark.}
Although previous studies have already provided decentralized online-to-nonconvex conversions, their algorithms and analyses are limited to specific choices of D-OCO algorithms. 
To the best of our knowledge, Theorem \ref{thm:general-reduction} provides the first general decentralized online-to-nonconvex conversion that can be instantiated with any D-OCO algorithm satisfying the required regret and consensus guarantees. 
In fact, Theorem \ref{thm:general-reduction} reduces all challenges introduced by decentralization to the properties of the underlying D-OCO algorithm. 
More importantly, the problem dimension $d$ only appears in the last term of the above bound and is multiplied by the consensus error bound $\mathcal{C}_{T,K}$. 
This allows us to remove the polynomial dependence on $d$ by sufficiently tightening the consensus error of the underlying D-OCO algorithm.

Now, we proceed to instantiate our conversion with the D-OCO algorithm of
\citet{wan2024nearly}. It is called accelerated decentralized
follow-the-regularized-leader (AD-FTRL), and its detailed procedure is summarized
in Algorithm~\ref{alg:dftrl}, where the superscript $k$ in the output
$\Delta_{t,i}$ and feedback $\g_{t,i}$ is omitted for brevity.
The key idea of this algorithm is to approximate the cumulative average gradient
via multiple accelerated gossip steps per round \citep{Acc_Gossip}, i.e., by
maintaining $\mathbf{q}_{t,i}$, and generate each local decision by applying the
classical follow-the-regularized-leader (FTRL) algorithm
\citep{hazan2016introduction} with the approximated gradient information, i.e.,
Step~$4$ of Algorithm~\ref{alg:dftrl}. Moreover, there exists a minor yet important difference between
Algorithm~\ref{alg:dftrl} and the version originally provided by
\citet{wan2024nearly}. To be precise, in the standard D-OCO setting, the number
of communication steps is commonly restricted to one per round. To implement
accelerated gossip under this communication constraint,
\citet{wan2024nearly} further employ a blocking update mechanism, i.e., local
decisions are updated only once every several rounds. As a result, both the
regret and consensus error bounds have a polynomial dependence on
$\gamma^{-1}$, which is undesirable for our application. Fortunately, the DSO
setting does not impose such a communication constraint. Therefore, we remove
the blocking update mechanism and establish the following guarantee.
\begin{algorithm}[t]
\caption{AD-FTRL}
\label{alg:dftrl}
\begin{algorithmic}[1]
\STATE \textbf{Input:} Decision set $\mathcal{K}$, learning rate $\eta$,
gossip matrix $P$, gossip rounds $R$, and mixing ratio $\theta$
\STATE Initialize
$\mathbf{q}_{0,i}=\mathbf{h}_{0,i}=\ze$ for every $i\in[n]$
\FOR{$t=1,\ldots,T$}
  \STATE Output
  $\Delta_{t,i}=\argmin_{\x\in\mathcal{K}}
  \{\langle\mathbf{q}_{t-1,i},\x\rangle+\|\x\|^2/(2\eta)\}$
  \STATE Receive $\g_{t,i}$ and set
  $\mathbf{q}_{t,i}^{0}=\mathbf{q}_{t-1,i}+\g_{t,i}$ and
  $\mathbf{q}_{t,i}^{-1}=\mathbf{h}_{t-1,i}+\g_{t,i}$ for all $i\in[n]$
  \FOR{$r=0,\ldots,R-1$}
    \STATE Set $\mathbf{q}_{t,i}^{r+1}
    =(1+\theta)\sum_{j=1}^nP_{ij}\mathbf{q}_{t,j}^{r}
    -\theta\mathbf{q}_{t,i}^{r-1}$ for all $i\in[n]$
  \ENDFOR
  \STATE Set $\mathbf{q}_{t,i}=\mathbf{q}_{t,i}^{R}$ and
  $\mathbf{h}_{t,i}=\mathbf{q}_{t,i}^{R-1}$ for all $i\in[n]$
\ENDFOR
\end{algorithmic}
\end{algorithm}

\begin{prop}
\label{prop:dftrl}
Suppose that Assumptions~\ref{assum:p} and \ref{assum:oracle} hold. For any $\zeta>0$, if Algorithm~\ref{alg:dftrl} invoked in Algorithm~\ref{alg:donc} is run with $\mathcal{K}=\{\x\in\R^d:\|\x\|\leq D\}$, $\eta=D/(G\sqrt T)$ and
\begin{equation}
  R=\left\lceil
  \frac{\sqrt2\log(1+2\sqrt{14}\,n/\min\{1,\zeta\sqrt T\})}
       {(\sqrt2-1)\sqrt\gamma}
  \right\rceil,\quad
  \theta=\frac{1}{1+\sqrt{1-\sigma_2^2(P)}}
  \label{eq:gossip-rounds}
\end{equation}
then the two bounds in \eqref{two-definitions} can be chosen as $\mathcal{R}_T=2nDG\sqrt T$ and $\mathcal{C}_{T,K}=\zeta D$.
\end{prop}

\textbf{Remark.}
First, as discussed above, the regret and consensus error bounds of
Algorithm~\ref{alg:dftrl} do not depend on the spectral gap $\gamma$.
In contrast, it only affects the number of communication rounds in our application.
Second, the consensus error of Algorithm~\ref{alg:dftrl} can be
sufficiently tightened with only logarithmic additional communication.
This is an important property that has been overlooked by previous
decentralized online-to-nonconvex conversions.

By combining Theorem~\ref{thm:general-reduction} with Proposition~\ref{prop:dftrl}, we have the following corollary.
\begin{cor}
\label{thm:main-complexity}
Under Assumptions~\ref{assum:p}--\ref{assum:oracle}, for any $\delta,\epsilon\in(0,1]$, 
Algorithm~\ref{alg:donc} can ensure $\E\|\nabla f(\bar\w)\|_\delta\leq\epsilon$
by setting $T=\left\lceil\max\left\{1,64G^2\epsilon^{-2}\right\}\right\rceil$, $
  K=\left\lceil\max\left\{1,8(\Delta_f+L\delta)\delta^{-1}\epsilon^{-1}\right\}\right\rceil$, and invoking
Algorithm~\ref{alg:dftrl}  with the same input as in Proposition \ref{prop:dftrl}, where $\zeta=\min\{1/2,
  \epsilon/(4L(K+1)(1+c\sqrt d))\}$ and $D=\delta/(2T)$.
\label{eq:main-parameters-2}
Accordingly, $O(\delta^{-1}\epsilon^{-3})$ stochastic gradient queries and $\widetilde O(\gamma^{-1/2}\delta^{-1}\epsilon^{-3})$ communication rounds are required in total.
\end{cor}
\vspace{-1em}
\begin{proof}
Let $A=\Delta_f+L\delta$ for brevity. By combining \eqref{eq:general-reduction} with Proposition \ref{prop:dftrl} and $D=\delta/(2T)$, we have
\begin{equation}
  \E\|\nabla f(\bar\w)\|_\delta\leq \frac{3G}{\sqrt T}+\frac{2A}{\delta K}
    +L(K+1)\zeta+\frac{c\sqrt d L(KT+1)}{2T}\zeta.
    \label{eq:total-result}
\end{equation}
By the definitions of $T$ and $K$, we always have
\[
  T\geq1,\qquad T\geq\frac{64G^2}{\epsilon^2},\qquad
  K\geq1,\qquad K\geq\frac{8A}{\delta\epsilon}
\]
which implies that
\[
  \frac{3G}{\sqrt T}\leq\frac{3\epsilon}{8},
  \qquad \frac{2A}{\delta K}\leq\frac{\epsilon}{4}.
\]
In particular, if $64G^2/\epsilon^2\leq1$, then $T=1$ and
$G\leq\epsilon/8$, which gives the same first bound. If
$8A/(\delta\epsilon)\leq1$, then $K=1$ and
$2A/\delta\leq\epsilon/4$, which gives the same second bound.

For the remaining terms in \eqref{eq:total-result}, $T\geq1$ implies $(KT+1)/T\leq K+1$.
Moreover, the definition of $\zeta$ ensures
$\zeta\leq\epsilon/[4L(K+1)(1+c\sqrt d)]$. Therefore,
\begin{align*}
  L(K+1)\zeta+\frac{c\sqrt d L(KT+1)}{2T}\zeta
  &\leq L(K+1)(1+c\sqrt d)\zeta
  \leq\frac{\epsilon}{4}.
\end{align*}
By combining \eqref{eq:total-result} with the three bounds, we have
\[
  \E\|\nabla f(\bar\w)\|_\delta
  \leq\frac{3\epsilon}{8}+\frac{\epsilon}{4}
       +\frac{\epsilon}{4}
  =\frac{7\epsilon}{8}\leq\epsilon.
\]
To bound the total complexities, we notice that
\[
  T\leq1+\frac{64G^2}{\epsilon^2}
    \leq\frac{1+64G^2}{\epsilon^2},\qquad
  K\leq1+\frac{8A}{\delta\epsilon}
    \leq\frac{1+8(\Delta_f+L)}{\delta\epsilon}
\]
where we used $\delta,\epsilon\in(0,1]$. Moreover, due to $\sqrt T\geq1$ and $\sqrt T\geq8G/\epsilon$, we have
\begin{align*}
  \frac1{\zeta\sqrt T}
  &=\max\left\{\frac2{\sqrt T},
       \frac{4L(K+1)(1+c\sqrt d)}{\epsilon\sqrt T}\right\}\leq\max\left\{2,
       \frac{L(K+1)(1+c\sqrt d)}{2G}\right\}
   =O\!\left(\frac{\sqrt d}{\delta\epsilon}\right).
\end{align*}
By combining this inequality with 
\eqref{eq:gossip-rounds}, we have $R=O(\gamma^{-1/2}
       \log(1+n\sqrt d\delta^{-1}\epsilon^{-1}))$. Finally, it is easy to verify that the algorithm uses $KT=O(\delta^{-1}\epsilon^{-3})$ stochastic gradient queries and $RKT=\widetilde O(\gamma^{-1/2}\delta^{-1}\epsilon^{-3})$ communication
rounds. 
\end{proof}
\vspace{-1em}
\textbf{Remark.} From Corollary \ref{thm:main-complexity}, we have established the first dimension-free  $O(\delta^{-1}\epsilon^{-3})$
sample complexity and the first nearly dimension-free $\widetilde O(\gamma^{-1/2}\delta^{-1}\epsilon^{-3})$ communication complexity for finding
an $(\delta,\epsilon)$-Goldstein stationary point of DSO with nonsmooth nonconvex objectives in expectation.
Moreover, the $O(\delta^{-1}\epsilon^{-3})$
sample complexity has been proved to be optimal even in the non-distributed setting \citep{Cutkosky2023}.

\section{Conclusion and Future Work}
This paper proposes a decentralized online-to-nonconvex conversion for
nonsmooth nonconvex stochastic optimization. Although it builds
on existing techniques, including randomized smoothing and client sampling,
its key property is that consensus among both the local iterates and the
random interpolation points can be inherited from the update directions
generated by a D-OCO algorithm, with a shared random interpolation weight
ensuring consensus among the interpolation points. This yields a general
reduction guarantee in which the dimension-dependent term is controlled
by the consensus error of the D-OCO subroutine. By instantiating the
conversion with AD-FTRL, an existing D-OCO algorithm whose accelerated
gossip steps sufficiently reduce this error, we obtain
$O(\delta^{-1}\epsilon^{-3})$ sample complexity and
$\widetilde{O}(\gamma^{-1/2}\delta^{-1}\epsilon^{-3})$ communication
complexity for finding an $(\delta,\epsilon)$-Goldstein stationary point
in expectation. These results eliminate the polynomial dimension
dependence in existing decentralized first-order guarantees, with only
logarithmic dependence on the dimension in the communication cost.
Nonetheless, several questions remain open. First, it is unclear
whether our communication complexity is nearly optimal. Second, it may
be possible to generalize our conversion to the zeroth-order setting
and achieve similar improvements.


\bibliography{iclr2027_conference}
\bibliographystyle{iclr2027_conference}

\appendix
\section{Proof of Theorem~\ref{thm:general-reduction}}
Since we have used the randomized smoothing technique, according to \eqref{eq:smoothing-definition},  we first define the smoothed versions of the global objective $f(\x)$ and $\{f_i(\x)\}_{i\in[n]}$ as follows
\begin{equation}
f_{i,\delta'}(\x)=\E_{\u\sim\operatorname{Unif}(\mathcal{B}^d)}[f_i(\x+\delta^\prime \u)],\forall i\in[n],\quad f_{\delta^\prime }(\x)=\frac{1}{n}\sum_{i=1}^nf_{i,\delta^\prime}(\x)
\label{eq:smoothed-versions}
\end{equation}
where $\delta'=\delta/2$. By combining the $L$-Lipschitzness of each $f_i(\x)$ derived from Assumption~\ref{assum:lip} with Lemma~\ref{lem:smoothing},
we can derive three nice properties for each $f_{i,\delta'}(\x)$. From the second property of Lemma~\ref{lem:smoothing}, $f_{i,\delta'}(\x)$ is differentiable and $L$-Lipschitz. Then,
for any $i\in[n]$, $t\in[T]$, and $k\in[K]$, the fundamental theorem of calculus gives
\begin{equation}
\begin{split}
  f_{i,\delta^\prime}(\x_{t,i}^k)-f_{i,\delta^\prime}(\x_{t-1,i}^k)&=\int_0^1\left\langle
  \nabla f_{i,\delta^\prime}(\x_{t-1,i}^k+s(\x_{t,i}^k-\x_{t-1,i}^k)),
  \x_{t,i}^k-\x_{t-1,i}^k\right\rangle\,ds\\
  &=\int_0^1\left\langle
  \nabla f_{i,\delta^\prime}(\x_{t-1,i}^k+s\Delta_{t,i}^k),
  \Delta_{t,i}^k\right\rangle\,ds
\end{split}
\label{eq:well-behaved}
\end{equation} 
where the second equality is due to the update rule $\x_{t,i}^{k}=\x_{t-1,i}^{k}+\Delta_{t,i}^{k}$. 

Moreover, according to the construction of each $\g_{t,i}^k$ and Assumption \ref{assum:oracle}, it is easy to verify that
\begin{equation}
\E[\langle\g_{t,i}^k, \Delta_{t,i}^k\rangle]=\E[\langle \nabla f_i(\w_{t,i}^k+\delta'\u_{t,i}^k), \Delta_{t,i}^k\rangle]=\E[\langle \nabla f_{i,\delta'}(\x_{t-1,i}^k+s_t^k\Delta_{t,i}^k), \Delta_{t,i}^k\rangle]
\label{eq:unbiased}
\end{equation}
where the second equality is due to the definition in \eqref{eq:smoothed-versions} and $\w_{t,i}^{k}=\x_{t-1,i}^{k}+s_t^k\Delta_{t,i}^{k}$.
By combining \eqref{eq:well-behaved}, \eqref{eq:unbiased}, and $s_t^k\sim\operatorname{Unif}[0,1]$, we have
\begin{equation}
  \E[\langle\g_{t,i}^k, \Delta_{t,i}^k\rangle]=\E[f_{i,\delta^\prime}(\x_{t,i}^k)-f_{i,\delta^\prime}(\x_{t-1,i}^k)].
\label{eq:single-node-iterate}
\end{equation} 
Summing both sides of \eqref{eq:single-node-iterate} over $t\in[T]$, we have
\begin{equation}
\begin{split}
  \sum_{t=1}^T\E[\langle\g_{t,i}^k,\Delta_{t,i}^k\rangle]=\E\left[f_{i,\delta^\prime}(\x_{T,i}^k)
  -f_{i,\delta^\prime}(\x_{0,i}^k)\right]=\E\left[f_{i,\delta^\prime}(\x_{T,i}^k)
  -f_{i,\delta^\prime}(\x_{T,i}^{k-1})\right]
\end{split}  
\label{eq:random-interpolation}
\end{equation}
where the last equality is due to $\x_{0,i}^k=\x_{T,i}^{k-1}$.

Note that the above analysis is only a natural extension of the non-distributed analysis of \citet{Cutkosky2023} to each local node. 
Our main technical innovation is to establish a connection between these local guarantees and the global objective.
Specifically, by defining $\bar{\x}_t^k=(1/n)\sum_{j=1}^n\x_{t,j}^k$, the $L$-Lipschitzness of $f_{i,\delta'}(\x)$ yields
\begin{equation}
\begin{split}
&\E\left[f_{\delta'}(\bar\x_T^k)-f_{\delta'}(\bar\x_T^{k-1})\right]\\
=&\E\left[\frac{1}{n}\sum_{i=1}
^n\left(f_{i,\delta'}(\x_{T,i}^k)-f_{i,\delta'}(\x_{T,i}^{k-1})+f_{i,\delta'}(\bar\x_T^k)-f_{i,\delta'}(\x_{T,i}^k)+f_{i,\delta'}(\x_{T,i}^{k-1})-f_{i,\delta'}(\bar\x_T^{k-1})\right)\right]\\
\leq& \E\left[\frac{1}{n}\sum_{i=1}
^n\left(f_{i,\delta'}(\x_{T,i}^k)-f_{i,\delta'}(\x_{T,i}^{k-1})+L\|\bar\x_T^k-\x_{T,i}^{k}\|+L\|\bar\x_T^{k-1}-\x_{T,i}^{k-1}\|\right)\right]\\
=& \E\left[\frac{1}{n}\sum_{t=1}^T\sum_{i=1}
^n\langle\g_{t,i}^k,\Delta_{t,i}^k\rangle\right]+\E\left[\frac{L}{n}\sum_{i=1}
^n(\|\bar\x_T^k-\x_{T,i}^{k}\|+\|\bar\x_T^{k-1}-\x_{T,i}^{k-1}\|)\right]
\end{split}
\label{eq:epoch-global-decrease}
\end{equation}
where the last equality is due to \eqref{eq:random-interpolation}.

To bound the first term on the right-hand side of \eqref{eq:epoch-global-decrease}, we make the following decomposition
\begin{equation}
\E\left[\frac{1}{n}\sum_{t=1}^T\sum_{i=1}
^n\langle\g_{t,i}^k,\Delta_{t,i}^k\rangle\right]=\E\left[\frac{1}{n}\sum_{t=1}^T\sum_{i=1}
^n\langle\g_{t,i}^k,\Delta_{t,i}^k-\v^k\rangle\right]+\E\left[\frac{1}{n}\sum_{t=1}^T\sum_{i=1}
^n\langle\g_{t,i}^k,\v^k\rangle\right]
\label{eq:decomposition1}
\end{equation}
for any $\v^k\in\R^d$. With $\bar{\Delta}_{t}^k=(1/n)\sum_{j=1}^n\Delta_{t,j}^k$, we have
\begin{equation}
\E\left[\sum_{t=1}^T\sum_{i=1}
^n\langle\g_{t,i}^k,\Delta_{t,i}^k-\v^k\rangle\right]=\E\left[\sum_{t=1}^T\sum_{i=1}
^n\langle\g_{t,i}^k,\bar{\Delta}_{t}^k-\v^k\rangle\right]+\E\left[\sum_{t=1}^T\sum_{i=1}
^n\langle\g_{t,i}^k,\Delta_{t,i}^k-\bar{\Delta}_{t}^k\rangle\right].
\label{eq:decomposition2}
\end{equation}
If $\v^k\in\mathcal{K}=\{\x\in\R^d:\|\x\|\leq D\}$, we can utilize the regret bound of the D-OCO algorithm to derive that
\begin{equation}
\E\left[\sum_{t=1}^T\sum_{i=1}
^n\langle\g_{t,i}^k,\bar{\Delta}_{t}^k-\v^k\rangle\right]=\frac{1}{n}\sum_{j=1}^n\E\left[\sum_{t=1}^T\sum_{i=1}
^n\langle\g_{t,i}^k,\Delta_{t,j}^k-\v^k\rangle\right]\leq \mathcal{R}_T
\label{eq:decomposition3}
\end{equation}
where $\mathcal{R}_T$ is defined in \eqref{two-definitions}. 

By further recalling $\mathcal{C}_{T,K}$ is defined in \eqref{two-definitions}, we also utilize the consensus error bound of the D-OCO algorithm to derive that
\begin{equation}
\begin{split}
\E\left[\sum_{t=1}^T\sum_{i=1}^n
\langle\g_{t,i}^k,\Delta_{t,i}^k-\bar\Delta_t^k\rangle\right]=&\E\left[\sum_{t=1}^T\sum_{i=1}^n
\langle\nabla f_{i,\delta'}(\w_{t,i}^k),
\Delta_{t,i}^k-\bar\Delta_t^k\rangle\right]\\
\leq& L\sum_{t=1}^T\sum_{i=1}^n
\E[\|\Delta_{t,i}^k-\bar\Delta_t^k\|]
\leq nLT\mathcal{C}_{T,K}
\end{split}
\label{eq:decomposition4}
\end{equation}
where the first inequality is due to the $L$-Lipschitzness of $f_{i,\delta'}(\x)$.
By combining \eqref{eq:decomposition1}, \eqref{eq:decomposition2}, \eqref{eq:decomposition3}, and \eqref{eq:decomposition4}, for $\v^k\in\mathcal{K}$, we have
\begin{equation}
\begin{split}
&\E\left[\frac{1}{n}\sum_{t=1}^T\sum_{i=1}
^n\langle\g_{t,i}^k,\Delta_{t,i}^k\rangle\right]\leq \E\left[\frac{1}{n}\sum_{t=1}^T\sum_{i=1}
^n\langle\g_{t,i}^k,\v^k\rangle\right]+\frac{\mathcal{R}_T}{n}+LT\mathcal{C}_{T,K}\\
=& \E\left[\frac{1}{n}\sum_{t=1}^T\sum_{i=1}
^n\langle\nabla f_{i,\delta^\prime }(\w_{t,i}^k),\v^k\rangle\right]
+\E\left[\frac{1}{n}\sum_{t=1}^T\sum_{i=1}
^n\left\langle\g_{t,i}^k-\nabla f_{i,\delta^\prime }(\w_{t,i}^k),\v^k\right\rangle\right]+\frac{\mathcal{R}_T}{n}+LT\mathcal{C}_{T,K}.
\end{split}
\label{eq:decomposition1-2}
\end{equation}
By choosing
\begin{equation*}
  \v^k=-D\frac{\sum_{t=1}^T\sum_{i=1}^n
  \nabla f_{i,\delta'}(\w_{t,i}^k)}
  {\|\sum_{t=1}^T\sum_{i=1}^n\nabla f_{i,\delta'}(\w_{t,i}^k)\|}
\end{equation*}
with $\v^k=\ze$ if the denominator is zero, we have $\v^k\in\mathcal{K}$ and
\begin{equation}
  \frac1n\sum_{t=1}^T\sum_{i=1}^n
  \langle\nabla f_{i,\delta'}(\w_{t,i}^k),\v^k\rangle
  =-\frac Dn\left\|\sum_{t=1}^T\sum_{i=1}^n
  \nabla f_{i,\delta'}(\w_{t,i}^k)\right\|.
  \label{eq:decomposition2-2}
\end{equation}
Moreover, let $\z_t^k=\frac1n\sum_{i=1}^n(\g_{t,i}^k-\nabla f_{i,\delta'}(\w_{t,i}^k))$ for brevity. 
Due to
$\v^k\in\mathcal{K}$, it is not hard to verify that
\begin{equation}
\begin{split}
  \E\left[\sum_{t=1}^T
  \langle\z_t^k,\v^k\rangle\right] \leq& D\E\left\|\sum_{t=1}^T\z_t^k\right\|\leq D\left(\E\left\|\sum_{t=1}^T\z_t^k\right\|^2\right)^{1/2}=D\left(\sum_{t=1}^T\E\|\z_t^k\|^2\right)^{1/2}\\
  = &D\left(\sum_{t=1}^T\E\left[\left\|\frac1n\sum_i\g_{t,i}^k\right\|^2-\left\|\frac1n\sum_i\nabla f_{i,\delta'}(\w_{t,i}^k)\right\|^2\right]\right)^{1/2}\\
  \leq &D\left(\sum_{t=1}^T\E\left[\left\|\frac1n\g_{t,i_t^k}^k\right\|^2\right]\right)^{1/2}\leq DG\sqrt{T}
\end{split}
\label{eq:decomposition3-2}
\end{equation}
where the second inequality is due to Jensen's inequality, and the last inequality is due to the construction of $\g_{t,i_t^k}^k$ and  Assumption \ref{assum:oracle}. 

By combining \eqref{eq:decomposition1-2}, \eqref{eq:decomposition2-2}, and
\eqref{eq:decomposition3-2} with \eqref{eq:epoch-global-decrease}, we have
\begin{equation}
\begin{split}
  \frac Dn\E\left\|\sum_{t=1}^T\sum_{i=1}^n
  \nabla f_{i,\delta'}(\w_{t,i}^k)\right\|
  \leq&\frac{\mathcal{R}_T}{n}+DG\sqrt T+LT\mathcal{C}_{T,K}+\E[f_{\delta'}(\bar\x_T^{k-1})-f_{\delta'}(\bar\x_T^k)]\\
  &+\frac Ln\sum_{i=1}^n\E\left[
  \|\x_{T,i}^k-\bar\x_T^k\|
  +\|\x_{T,i}^{k-1}-\bar\x_T^{k-1}\|\right].
\end{split}
\label{eq:epoch-gradient-bound}
\end{equation}
Taking the average of both sides over $k\in[K]$, we have
\begin{equation}
\begin{split}
  \frac{D}{nK}\sum_{k=1}^K\E\left\|\sum_{t=1}^T\sum_{i=1}^n
  \nabla f_{i,\delta'}(\w_{t,i}^k)\right\|
  \leq&\frac{\mathcal{R}_T}{n}+DG\sqrt T+LT\mathcal{C}_{T,K}+\frac{1}{K}\E[f_{\delta'}(\bar\x_T^{0})-f_{\delta'}(\bar\x_T^K)]\\
  &+\frac{L}{Kn}\sum_{k=1}^K\sum_{i=1}^n\E\left[
  \|\x_{T,i}^k-\bar\x_T^k\|
  +\|\x_{T,i}^{k-1}-\bar\x_T^{k-1}\|\right].
\end{split}
\label{eq:sum-epoch-gradient-bound}
\end{equation}
Due to $\bar\x_T^0=\ze$ and the first property of Lemma~\ref{lem:smoothing}, we have
\begin{equation}
f_{\delta'}(\bar\x_T^{0})-f_{\delta'}(\bar\x_T^k)=f_{\delta'}(\ze)-f_{\delta'}(\bar\x_T^k)\leq f(\ze)-f(\bar\x_T^k)+2L\delta^\prime\leq \Delta_f+L\delta.
\label{eq:initial-gap}
\end{equation}
To bound $\|\x_{T,i}^k-\bar\x_T^k\|$, we exploit the consensus error bound of the D-OCO algorithm to establish the following lemma.
\begin{lem}
\label{lem:consensus-inheritance}
For every $i\in[n]$, $t\in[T]$, and $k\in[K]$, Algorithm \ref{alg:donc} ensures
\[
  \E\|\x_{t,i}^k-\bar{\x}_t^k\|
  \leq ((k-1)T+t)\mathcal{C}_{T,K}.
\]
\end{lem}
From Lemma \ref{lem:consensus-inheritance}, it is easy to verify that
\begin{equation}
\begin{split}
\sum_{k=1}^K\sum_{i=1}^n\E\left[
  \|\x_{T,i}^k-\bar\x_T^k\|
  +\|\x_{T,i}^{k-1}-\bar\x_T^{k-1}\|\right]\leq nT\sum_{k=1}^K(2k-1)\mathcal{C}_{T,K}=nTK^2\mathcal{C}_{T,K}.
\end{split}
\label{eq:iterate-gap-bound}
\end{equation}

By combining \eqref{eq:sum-epoch-gradient-bound}, \eqref{eq:initial-gap}, and \eqref{eq:iterate-gap-bound}, we have
\begin{equation}
\begin{split}
  \frac{D}{nK}\sum_{k=1}^K\E\left\|\sum_{t=1}^T\sum_{i=1}^n
  \nabla f_{i,\delta'}(\w_{t,i}^k)\right\|\leq&\frac{\mathcal{R}_T}{n}+DG\sqrt T+LT\mathcal{C}_{T,K}+\frac{\Delta_f+L\delta}{K}+LKT\mathcal{C}_{T,K}.
\end{split}
\label{eq:sum-epoch-gradient-bound-simplied}
\end{equation}
Finally, we introduce the following lemma, which relates the above bound on these local gradients to the stationarity of the original objective.
\begin{lem}
\label{transfer-gradient-point}
Suppose that Assumption \ref{assum:lip} holds. Following the definitions in \eqref{two-definitions}, for any $\delta>0$, Algorithm~\ref{alg:donc} with
$\mathcal{K}=\{\x\in\R^d:\|\x\|\leq D\}$ and $D=\delta/(2T)$ satisfies
\[
\E[\|\nabla f(\bar{\w})\|_{\delta}]\leq \frac{1}{K}\E\left[\sum_{k=1}^K\left\|\frac{1}{nT}\sum_{t=1}^T\sum_{i=1}^n
 \nabla f_{i,\delta'}(\w_{t,i}^k)\right\|\right]+\frac{c\sqrt{d}L(KT+1)\mathcal{C}_{T,K}}{\delta}
\]
where $\delta^\prime=\delta/2$, and $c>0$ is the numerical constant in
Lemma~\ref{lem:smoothing}.
\end{lem}
It is easy to complete this proof by combining Lemma \ref{transfer-gradient-point} with \eqref{eq:sum-epoch-gradient-bound-simplied}.

\section{Proof of Lemma \ref{lem:consensus-inheritance}}
For any epoch $k\in[K]$ and iteration $t\in[T]$, the update
rule in Algorithm~\ref{alg:donc} and the definitions of $\bar{\x}_t^k$ and $\bar\Delta_t^k$ give
\begin{equation*}
\begin{split}
  \E\left[\|\x_{t,i}^k-\bar\x_t^k\|\right]
  & =\E\left[\left\|\x_{t-1,i}^k+\Delta_{t,i}^k
  -\frac1n\sum_{j=1}^n(\x_{t-1,j}^k+\Delta_{t,j}^k)\right\|\right]\\
  & =\E\left[\left\|(\x_{t-1,i}^k-\bar\x_{t-1}^k)
  +(\Delta_{t,i}^k-\bar\Delta_t^k)\right\|\right]\\
  &\leq\E\left[\|\x_{t-1,i}^k-\bar\x_{t-1}^k\|\right]
  +\E\left[\|\Delta_{t,i}^k-\bar\Delta_t^k\|\right].
\end{split}
\end{equation*}
Applying this inequality recursively within
epoch $k$, we obtain
\begin{equation}
\begin{split}
  \E\left[\|\x_{t,i}^k-\bar\x_t^k\|\right]
  &\leq\E\left[\|\x_{0,i}^k-\bar\x_0^k\|\right]
  +\sum_{s=1}^t\E\left[\|\Delta_{s,i}^k-\bar\Delta_s^k\|\right]\\
  &\leq\E\left[\|\x_{0,i}^k-\bar\x_0^k\|\right]
  +t\mathcal{C}_{T,K}
\end{split}
\label{eq:iterate-gap-1}
\end{equation}
where the last inequality is due to the definition of $\mathcal{C}_{T,K}$.
For any $k\geq 2$, from \eqref{eq:iterate-gap-1}, it is easy to verify that
\begin{equation}
\begin{split}
  \E\left[\|\x_{0,i}^k-\bar\x_0^k\|\right]
  &=\E\left[\|\x_{T,i}^{k-1}-\bar\x_T^{k-1}\|\right]\leq\E\left[\|\x_{0,i}^{k-1}-\bar\x_0^{k-1}\|\right]
  +T\mathcal{C}_{T,K}\\
  &\leq\E\left[\|\x_{0,i}^{1}-\bar\x_0^{1}\|\right]
  +(k-1)T\mathcal{C}_{T,K}=(k-1)T\mathcal{C}_{T,K}
\end{split}
\label{eq:iterate-gap-2}
\end{equation}
where the last equality is due to the common initialization
$\x_{0,i}^1=\x_{T,i}^0=\ze$ for all $i\in[n]$. Moreover, for $k=1$, we have
\begin{equation}
\begin{split}
  \E\left[\|\x_{0,i}^k-\bar\x_0^k\|\right]=\E\left[\|\x_{0,i}^1-\bar\x_0^1\|\right]=0=(k-1)T\mathcal{C}_{T,K}
\end{split}
\label{eq:iterate-gap-3}
\end{equation}
By substituting \eqref{eq:iterate-gap-2} and \eqref{eq:iterate-gap-3} into \eqref{eq:iterate-gap-1}, we finally have
\[
  \E\left[\|\x_{t,i}^k-\bar\x_t^k\|\right]
  \leq(k-1)T\mathcal{C}_{T,K}+t\mathcal{C}_{T,K}
  =((k-1)T+t)\mathcal{C}_{T,K}.
\]

\section{Proof of Lemma \ref{transfer-gradient-point}}
From \eqref{eq:smoothing-stationarity} derived from Lemma~\ref{lem:smoothing-goldstein}, we only need to 
derive an upper bound on $\E[\|\nabla f_{\delta^\prime}(\bar{\w})\|_{\delta^\prime}]$, where $\delta^\prime=\delta/2$.
Due to $\bar{\w}\sim\operatorname{Unif}\{\bar{\w}^{1},\ldots,\bar{\w}^{K}\}$, it is easy to verify that
\begin{equation}
\E[\|\nabla f_{\delta^\prime}(\bar{\w})\|_{\delta^\prime}]=\frac{1}{K}\E\left[\sum_{k=1}^K\|\nabla f_{\delta^\prime}(\bar{\w}^k)\|_{\delta^\prime}\right].
\label{eq:goal-1}
\end{equation}
Let $\bar\w_t^k=(1/n)\sum_{i=1}^n\w_{t,i}^k$. To utilize \eqref{eq:goal-1}, we need to bound $\|\bar{\w}^k-\bar\w_t^k\|$. Specifically, for any $1\leq t^\prime <t\leq T$, we have
\begin{equation}
\begin{split}
\bar\w_t^k-\bar\w_{t^\prime}^k=&(\bar\x_{t-1}^k+s_t^k\bar\Delta_t^k)-(\bar\x_{t^\prime-1}^k+s_{t^\prime}^k\bar\Delta_{t^\prime}^k)= (1-s_{t^\prime}^k)\bar\Delta_{t^\prime}^k+\sum_{r=t^\prime+1}^{t-1}\bar\Delta_r^k+s_t^k\bar\Delta_t^k
\end{split}
\label{eq:goal-2}
\end{equation}
where the first equality is due to $\w_{t,i}^{k}=\x_{t-1,i}^{k}+s_t^k\Delta_{t,i}^{k}$ and the second one is due to $ \bar\x_t^k=\bar\x_{t-1}^k+\bar\Delta_t^k$.

From \eqref{eq:goal-2}, it is natural to derive that
\begin{equation}
\begin{split}
\|\bar\w_t^k-\bar\w_{t^\prime}^k\|\leq (1-s_{t^\prime}^k)\|\bar\Delta_{t^\prime}^k\|+\sum_{r=t^\prime+1}^{t-1}\|\bar\Delta_r^k\|+s_t^k\|\bar\Delta_t^k\|\leq (t-t^\prime+1)D\leq TD
\end{split}
\label{eq:goal-3}
\end{equation}
where the second inequality is due to $\Delta_{t,i}^k\in\mathcal{K}=\{\x\in\R^d:\|\x\|\leq D\},\forall i\in[n],t\in[T],k\in[K]$.

By combining the definition of $\bar\w^k$ with \eqref{eq:goal-3}, we have
\begin{equation}
  \|\bar\w_t^k-\bar\w^k\|
  =\left\|\frac1T\sum_{t^\prime=1}^T(\bar\w_t^k-\bar\w_{t^\prime}^k)\right\|
  \leq\frac1T\sum_{t^\prime=1}^T\|\bar\w_t^k-\bar\w_{t^\prime}^k\|
  \leq TD=\delta^\prime
  \label{eq:interpolation-output-distance}
\end{equation}
where the last inequality is due to $D=\delta/(2T)$. Recall that $\|\nabla f_{\delta^\prime}(\bar{\w}^k)\|_{\delta^\prime}=\min\{\|\g\|:\g\in\partial_{\delta^\prime} f_{\delta^\prime}(\bar{\w}^k)\}$. According to the definition of 
$\partial_{\delta^\prime} f_{\delta^\prime}(\bar{\w}^k)$, we only need to consider the Clarke subdifferential of $f_{\delta'}$ at $\{\bar\w_{t}^k\}_{t\in[T]}$. 
Moreover, due to the third property of Lemma~\ref{lem:smoothing}, both $f_{i,\delta'}$ and their average $f_{\delta'}$ have
$(c\sqrt{d}L/\delta')$-Lipschitz gradients for some constant $c>0$. Therefore, the Clarke subdifferential
of $f_{\delta^\prime}$ at any point consists of its gradient alone. From \eqref{eq:interpolation-output-distance}, we have
\begin{equation}
\begin{split}
  \frac1T\sum_{t=1}^T\nabla f_{\delta'}(\bar\w_t^k)
  &\in\operatorname{conv}(\{\nabla f_{\delta'}(\bar\w_t^k):t\in[T]\})\subseteq\operatorname{conv}\!\left(
  \bigcup_{\|\y-\bar\w^k\|\leq\delta^\prime}\partial f_{\delta'}(\y)\right)
  =\partial_{\delta^\prime}f_{\delta'}(\bar\w^k).
\end{split}
\label{eq:average-gradient-membership}
\end{equation}
By combining \eqref{eq:goal-1} with \eqref{eq:average-gradient-membership}, we have
\begin{equation}
\E[\|\nabla f_{\delta^\prime}(\bar{\w})\|_{\delta^\prime}]\leq \frac{1}{K}\E\left[\sum_{k=1}^K\left\|\frac1T\sum_{t=1}^T\nabla f_{\delta'}(\bar\w_t^k)\right\|\right].
\label{eq:goal-4}
\end{equation}
Moreover, for any $i\in[n]$, $t\in[T]$, and $k\in[K]$, we notice that
\begin{equation}
\begin{split}
    \E[\|\w_{t,i}^k-\bar{\w}_t^k\|]=&\E\left[\left\|\x_{t-1,i}^{k}+s_t^k\Delta_{t,i}^{k}-\frac{1}{n}\sum_{j=1}^n
    (\x_{t-1,j}^{k}+s_t^k\Delta_{t,j}^{k})\right\|\right]\\
    =&\E\left[\left\|\x_{t-1,i}^{k}-\bar{\x}_{t-1}^k+s_t^k\Delta_{t,i}^{k}-s_t^k\bar{\Delta}_t^k\right\|\right]\\
    \leq &\E\left[\|\x_{t-1,i}^{k}-\bar{\x}_{t-1}^k\|\right]+\E\left[\|\Delta_{t,i}^{k}-\bar{\Delta}_t^k\|\right]\\
    \leq &((k-1)T+t)\mathcal{C}_{T,K}
\end{split}
\label{eq:solution-gap}
\end{equation}
where the last inequality is due to Lemma \ref{lem:consensus-inheritance} and the definition of $\mathcal{C}_{T,K}$.

Therefore, we have
\begin{equation}
\begin{split}
 &\E\left[\left\|\nabla f_{\delta'}(\bar\w_t^k)-\frac{1}{n}\sum_{i=1}^n
 \nabla f_{i,\delta'}(\w_{t,i}^k)\right\|\right]= \E\left[\left\|\frac{1}{n}\sum_{i=1}^n
\left(\nabla f_{i,\delta'}(\bar\w_t^k)- \nabla f_{i,\delta'}(\w_{t,i}^k)\right)\right\|\right]\\
\leq& \frac{c\sqrt{d}L}{n\delta^\prime}\sum_{i=1}^n\E\left[\left\|\bar\w_t^k-\w_{t,i}^k\right\|\right]\leq \frac{c\sqrt{d}L((k-1)T+t)\mathcal{C}_{T,K}}{\delta^\prime}\\
\end{split}
\label{eq:solution-gradient-gap}
\end{equation}
where the last inequality is due to \eqref{eq:solution-gap}.

By combining \eqref{eq:goal-4} with \eqref{eq:solution-gradient-gap}, we have
\begin{equation}
\begin{split}
\E[\|\nabla f_{\delta^\prime}(\bar{\w})\|_{\delta^\prime}]\leq& \frac{1}{K}\E\left[\sum_{k=1}^K\left\|\frac{1}{nT}\sum_{t=1}^T\sum_{i=1}^n
 \nabla f_{i,\delta'}(\w_{t,i}^k)\right\|\right]+\frac{c\sqrt{d}L\mathcal{C}_{T,K}}{KT\delta^\prime}\sum_{k=1}^K\sum_{t=1}^T((k-1)T+t)\\
 \leq &\frac{1}{K}\E\left[\sum_{k=1}^K\left\|\frac{1}{nT}\sum_{t=1}^T\sum_{i=1}^n
 \nabla f_{i,\delta'}(\w_{t,i}^k)\right\|\right]+\frac{c\sqrt{d}L(KT+1)\mathcal{C}_{T,K}}{\delta}
\end{split}
\label{eq:goal-5}
\end{equation}
Finally, we can complete this proof by combining \eqref{eq:smoothing-stationarity} with \eqref{eq:goal-5}.

\section{Proof of Proposition~\ref{prop:dftrl}}
\label{app:dftrl-proof}
The original analysis of \citet{wan2024nearly} assumes a deterministic bound on the
local gradient norms, i.e., $\|\g_{t,i}\|$. In our setting, Assumption~\ref{assum:oracle} only
provides a second-moment bound, so their guarantee cannot be applied directly. Fortunately, 
it is not hard to extend their analysis to our setting by using an expected bound regarding 
the local gradient norms. 

For brevity, we fix an epoch $k$ and omit the epoch superscript. Let $\mathbf v_t$ be the unscaled stochastic gradient queried at round $t$ of the epoch $k$, i.e.,
\[
  \mathbf v_t=\nabla F_{i_t}
  (\w_{t,i_t}+(\delta/2)\u_{t,i_t};\xi_{t,i_t}).
\]
Due to client sampling, $\g_{t,i_t}=n\mathbf v_t$ and $\g_{t,i}=\ze$ for $i\ne i_t$,
so $\bar\g_t=(1/n)\sum_{i=1}^n\g_{t,i}=\mathbf v_t$. From Assumption~\ref{assum:oracle},
it is easy to verify that
\begin{equation}
  \E\|\bar\g_t\|^2
  =\E\|\mathbf v_t\|^2\leq G^2, \quad \E\|\g_{t,i}\|^2
  =n^2\E[\mathbf 1_{\{i_t=i\}}\|\mathbf v_t\|^2]
    \leq nG^2,i\in[n].
\label{eq:app-global-moment}
\end{equation}
Let
$\mathbf s_t=\sum_{s=1}^t\bar\g_s$ with $\mathbf s_0=\ze$. Following \citet{wan2024nearly}, we 
define the virtual global FTRL decision for any $t\in[T+1]$ by
\[
  \mathbf{y}_t=\argmin_{\x\in\mathcal{K}}
  \left\{\langle\mathbf{s}_{t-1},\x\rangle
    +\frac{\|\x\|^2}{2\eta}\right\}
  =\Pi_{\mathcal{K}}(-\eta\mathbf{s}_{t-1}),
\]
where $\Pi_{\mathcal{K}}(\cdot)$ denotes Euclidean projection onto $\mathcal{K}$. In Algorithm \ref{alg:dftrl},
the local decision can be similarly rewritten as $\Delta_{t,i}=\Pi_{\mathcal{K}}(-\eta\mathbf{q}_{t-1,i})$. 
Then, the difference between the local and virtual decisions can be bounded by combining $\|\mathbf{q}_{t-1,i}-\mathbf{s}_{t-1}\|$ with the following lemma.
\begin{lem}[Lemma 5 in \citet{DADO2011}]
\label{lem-stab}
For any $\eta>0$ and $\u,\v\in\R^d$, it holds that
\[
  \|\Pi_{\mathcal{K}}(-\eta\u)
    -\Pi_{\mathcal{K}}(-\eta\v)\|
  \leq\eta\|\u-\v\|.
\]
\end{lem}
Let $c_0=1-1/\sqrt2$, $\rho=1-c_0\sqrt\gamma\in(0,1)$, and $\beta=\rho^R\in(0,1)$. Following the proof of Lemma 2 in \citet{wan2024nearly}, for any $t\geq 2$, it is not hard to verify that
\begin{equation}
\begin{split}
\|\mathbf{q}_{t-1,i}-\mathbf{s}_{t-1}\|\leq \sqrt{14}\sum_{\tau=1}^{t-1}\rho^{(t-\tau)R}\sqrt{\sum_{j=1}^n\|\g_{\tau,j}\|^2}.
\end{split}
\label{eq:general-consensus-bound}
\end{equation}
Taking expectations in \eqref{eq:general-consensus-bound}, for any $t\geq 2$, we obtain
\begin{equation}
\begin{split}
  \E\|\mathbf q_{t-1,i}-\mathbf s_{t-1}\|
  &\leq\sqrt{14}\sum_{\tau=1}^{t-1}
    \rho^{(t-\tau)R}
    \E\left[\left(\sum_{j=1}^n\|\g_{\tau,j}\|^2\right)^{1/2}\right]\\
  &\leq\sqrt{14}\sum_{\tau=1}^{t-1}
    \rho^{(t-\tau)R}
    \left(\sum_{j=1}^n\E\|\g_{\tau,j}\|^2\right)^{1/2}\\
  &\leq\sqrt{14}\,nG\sum_{\tau=1}^{t-1}\beta^{t-\tau}
  \leq\frac{\sqrt{14}\,nG\beta}{1-\beta}
  \label{eq:app-state-error}
\end{split}
\end{equation}
where the second inequality is due to Jensen's inequality and the third one is due to \eqref{eq:app-global-moment}. 
For $t=1$, the same bound holds trivially since $\mathbf q_{0,i}=\mathbf s_0=\ze$.

By combining \eqref{eq:app-state-error} with Lemma~\ref{lem-stab}, for any $i\in[n]$ and $t\in[T]$, we have
\begin{equation}
  \E\|\Delta_{t,i}-\mathbf y_t\|
  \leq\eta\E\|\mathbf q_{t-1,i}-\mathbf s_{t-1}\|
  \leq \frac{\sqrt{14}\eta nG\beta}{1-\beta}.
  \label{eq:app-decision-error}
\end{equation}
Due to $\bar\Delta_t=(1/n)\sum_{j=1}^n\Delta_{t,j}$ and \eqref{eq:app-decision-error}, it is easy to verify that
\begin{equation}
\begin{split}
  \E\|\Delta_{t,i}-\bar\Delta_t\|
  \leq\E\|\Delta_{t,i}-\mathbf y_t\|
    +\frac1n\sum_{j=1}^n\E\|\Delta_{t,j}-\mathbf y_t\|\leq\frac{2\sqrt{14}\eta nG\beta}{1-\beta}.
\end{split}
\label{eq:app-maximum-consensus-bound}
\end{equation}
By the choice of $R$ in \eqref{eq:gossip-rounds}, we have
\begin{equation}
  \beta=\rho^R\leq\exp(-c_0\sqrt\gamma R)
  \leq\frac{\min\{1,\zeta\sqrt T\}}
  {\min\{1,\zeta\sqrt T\}+2\sqrt{14}\,n}.
  \label{eq:beta-up}
\end{equation}
By substituting this bound and $\eta=D/(G\sqrt T)$ into
\eqref{eq:app-maximum-consensus-bound}, we have
\[
  \max_{i,t}\E\|\Delta_{t,i}-\bar\Delta_t\|
  \leq\frac{2\sqrt{14}\,nD}{\sqrt T}\frac{\beta}{1-\beta}
  \leq\frac{D}{\sqrt T}\min\{1,\zeta\sqrt T\}
  \leq\zeta D.
\]
Note that such an upper bound holds for every epoch $k\in[K]$. Thus, we can choose $\mathcal C_{T,K}=\zeta D$, as stated in
Proposition~\ref{prop:dftrl}.

It remains to bound the regret. Let $\operatorname{Reg}_{T,i}$ denote the
realized regret inside the expectation in \eqref{two-definitions}, with the
epoch index suppressed, i.e.,
\[
\operatorname{Reg}_{T,i}=\sum_{t=1}^T\sum_{j=1}^n\langle\g_{t,j},\Delta_{t,i}\rangle-\min_{\u\in\mathcal{K}}\sum_{t=1}^T\sum_{j=1}^n\langle\g_{t,j},\u\rangle.
\]
From the proof of Theorem~1 in \citet{wan2024nearly}, for every $\u\in\mathcal K$,
the virtual global FTRL decision ensures that
\begin{equation}
  \sum_{t=1}^T\langle\bar\g_t,\mathbf{y}_t-\u\rangle
  \leq\frac{\|\u\|^2}{2\eta}
   +\eta\sum_{t=1}^T\|\bar\g_t\|^2
  \leq\frac{D^2}{2\eta}
   +\eta\sum_{t=1}^T\|\bar\g_t\|^2.
  \label{eq:app-pathwise-ftrl}
\end{equation}
Let $\u^\ast\in\argmin_{\u\in\mathcal{K}} \sum_{t=1}^T\sum_{j=1}^n\langle\g_{t,j},\u\rangle$. From \eqref{eq:app-pathwise-ftrl}, it is not hard to verify that
\begin{equation}
\frac{1}{n}\operatorname{Reg}_{T,i}= \sum_{t=1}^T\langle\bar{\g}_{t},\Delta_{t,i}-\u^\ast\rangle\leq \frac{D^2}{2\eta}+\eta\sum_{t=1}^T\|\bar\g_t\|^2
   +\sum_{t=1}^T\langle\bar\g_t,\Delta_{t,i}-\mathbf{y}_t\rangle.
   \label{eq:reg-1}
\end{equation}
Let $B_R=\sqrt{14}\,nG\beta/(1-\beta)$ for brevity. From \eqref{eq:app-global-moment} and Jensen's inequality,
we have
\[
  \E\langle\bar\g_t,\Delta_{t,i}-\mathbf y_t\rangle
  \leq G\,\E\|\Delta_{t,i}-\mathbf y_t\|
  \leq\eta GB_R
\]
where the second inequality is due to \eqref{eq:app-decision-error}.

Then, taking expectations on both sides of \eqref{eq:reg-1} and combining \eqref{eq:app-global-moment}, it is not hard to derive that
\[
  \frac1n\E[\operatorname{Reg}_{T,i}]
  \leq\frac{D^2}{2\eta}+\eta TG^2+\eta TGB_R.
\] 
From \eqref{eq:beta-up}, we have
\[
 \frac{B_R}{G}\leq\frac{\min\{1,\zeta\sqrt T\}}{2}.
\]
By combining the above two inequalities with $\eta=D/(G\sqrt T)$, we have
\begin{align*}
  \E[\operatorname{Reg}_{T,i}]
  &\leq nDG\sqrt T\left(\frac32+\frac{\min\{1,\zeta\sqrt T\}}{2}\right)
  \leq2nDG\sqrt T
\end{align*}
Note that such an upper bound holds for every epoch $k\in[K]$. Thus, we can choose $\mathcal R_T=2nDG\sqrt T$, as stated in
Proposition~\ref{prop:dftrl}.


\end{document}